\documentclass[10pt,english,reqno]{amsart}
\usepackage[T1]{fontenc}
\usepackage[latin9]{inputenc}
\usepackage{amsthm}
\usepackage{amssymb}

\makeatletter
\numberwithin{equation}{section}
\numberwithin{figure}{section}
  \theoremstyle{plain}
  \newtheorem*{conjecture*}{Conjecture}
\theoremstyle{plain}
\newtheorem{thm}{Theorem}
  \theoremstyle{definition}
  \newtheorem{defn}[thm]{Definition}
  \theoremstyle{plain}
  \newtheorem{prop}[thm]{Proposition}

\usepackage[matrix,arrow,tips,curve,ps]{xy}\SelectTips{cm}{10}
\usepackage{enumerate}%\input xypic
\usepackage{comment} 

\theoremstyle{definition}

\numberwithin{Theorem}{section} \numberwithin{equation}{section}
\makeatother

\usepackage{babel}

\begin{document}

\newcommand{\Gal}{{\rm Gal}} 
\newcommand{\cha}{{\rm char. }} 
\newcommand{\DIM}{{\rm dim}}

\newtheorem{Assumption}[thm]{Assumption}

\title[On the Tate conjecture for Fano surfaces]{On the Tate conjecture for the Fano surfaces of cubic threefolds}

\addtolength{\textwidth}{20mm}
\addtolength{\hoffset}{-0mm} 
\addtolength{\textheight}{20mm}
\addtolength{\voffset}{-0mm}

\author{Xavier Roulleau}
\begin{abstract}
A Fano surface of a smooth cubic threefold $X\hookrightarrow\mathbb{P}^{4}$
parametrizes the lines on $X$. In this note, we prove that a Fano
surface satisfies the Tate conjecture over a field of finite type
over the prime field and characteristic not $2$.
\end{abstract}
\maketitle

\section{Introduction.}

Let $k$ be a field of finite type over the prime field and let $\ell$
be a prime integer, prime to the characteristic. We denote by $\bar{k}$
an algebraic closure of $k$ and by $G$ the Galois group $\Gal(\bar{k}/k)$.
Let $X$ be a geometrically connected smooth projective variety over
$k$ and $\bar{X}:=X\times_{k}\bar{k}$. We denote by $A^{1}(X)$
the $\mathbb{Q}_{\ell}$-span of the images of the divisor classes
defined over $k$ in the (twisted) étale cohomology group $H^{2}(\bar{X},\mathbb{Q}_{\ell}(1))$.
The group $G$ acts on $H^{2}(\bar{X},\mathbb{Q}_{\ell}(1))$ and
fixes the subspace $A^{1}(X)$. The Tate conjecture for divisors is: 
\begin{conjecture*}
(Tate Conjecture \cite{Tate1}). We have $A^{1}(X)=H^{2}(\bar{X},\mathbb{Q}_{\ell}(1))^{G}$. 
\end{conjecture*}
The Tate conjecture for Abelian varieties has been proved by Tate
\cite{Tate4} for finite fields, by Zarhin \cite{Zarhin} in characteristic
$>2$, by Mori \cite{Mori} in characteristic $2$, and by Faltings
\cite{Faltings} for fields of characteristic 0. We know a few cases
of surfaces that satisfy the Tate conjecture (K3 surfaces, product
of two curves, some Picard modular surfaces...)

In this note, we prove the Tate conjecture for another family of surfaces.
Let us suppose that the field $k$ has moreover characteristic $\not=2$
and let $X\hookrightarrow\mathbb{P}_{/k}^{4}$ be a smooth cubic hypersurface.
The variety that parametrizes the lines on $X$ is a smooth projective
surface defined over $k$ called the \emph{Fano surface of lines}
of $X$. This surface $S$ is minimal of general type and has invariants:\[
c_{1}^{2}=45,\, c_{2}=27,\, b_{1}=10,\, b_{2}=45.\]
We obtain the following result:
\begin{thm}
\label{thm:principal result The-Tate-conjecture} The Tate conjecture
holds for the surface $S$. 
\end{thm}
For the proof we use the fact that the Fano surface is contained in
its $5$ dimensional Albanese variety $A$ and has class $\frac{1}{3!}\Theta^{3}$
where $\Theta$ is a principal polarization. Using then the Hard Lefschetz
Theorem, the Poincaré Duality, and the equality $b_{2}(\bar{S})=b_{2}(\bar{A})$
of Betti numbers, we obtain that the natural map\[
H^{2}(\bar{A},\mathbb{Q}_{\ell}(1))\to H^{2}(\bar{S},\mathbb{Q}_{\ell}(1))\]
is an isomorphism. Therefore the second étale cohomology group of
$S$ is essentially the same as the second étale cohomology group
of the Abelian variety $A$, for which the Tate conjecture is known.

To the knowledge of the author, Abelian surfaces and Fano surfaces
are the only known surfaces $S$ such that there is an isomorphism
between the second étale cohomology groups of $S$ and its Albanese
variety.

\textbf{Acknowledgements.} The author gratefully thanks Marc-Hubert
Nicole for his encouragements to work on this question. He thanks
also the Editor and the Referee for their numerous suggestions to
improve the exposition of this paper, and Stavros Papadakis, Alessandra
Sarti, Matthias Schütt for their comments.\\
 The author has been supported under FCT grant SFRH/BPD/72719/2010
and project Géometria Algébrica PTDC/MAT/099275/2008.

\section{The proof}

Let $k$ be a field finitely generated over its prime field, and let
$\bar{k}$ be the algebraic closure of $k$. Recall \cite[Theorem 11.1]{MilneEtale}
that for a smooth $n$-dimensional projective variety $Z$ over $\bar{k}$,
there exists a canonical isomorphism $\eta_{Z}:H^{2n}(Z,\mathbb{Q}_{\ell}(n))\to\mathbb{Q}_{\ell}$
sending the class of a closed point to $1$. Let $A$ be an Abelian
variety of dimension $n\geq2$ defined over $k$. 
\begin{defn}
We say that a $2$-dimensional cycle $W$ on $A$ is \emph{non-degenerate}
if the $\mathbb{Q}_{\ell}$-bilinear form: \[
\begin{array}{c}
Q_{W}:\\
\\\end{array}\begin{array}{ccc}
H^{2}(\bar{A},\mathbb{Q}_{\ell}(1))\times H^{2}(\bar{A},\mathbb{Q}_{\ell}(1)) & \to & \mathbb{Q}_{\ell}\\
(x,y) & \to & \eta_{\bar{A}}(x\cdot W\cdot y)\end{array}\]
 is non-degenerate, where we consider the cycle $W$ in $H^{2n-4}(\bar{A},\mathbb{Q}_{\ell}(n-2))$
and $\cdot$ denotes the cup product. 
\end{defn}
An example of a non-degenerate cycle is:
\begin{prop}
Let $\Theta$ be an ample divisor on $A$. The cycle $\frac{1}{(n-2)!}\Theta^{n-2}$
is non-degenerate.\end{prop}
\begin{proof}
By the Hard Lefschetz Theorem of Deligne \cite[Théorème 4.1.1]{Deligne},
the cup product induced by $\Theta^{n-2}$ induces an isomorphism
between $H^{2}(\bar{A},\mathbb{Q}_{\ell}(1))$ and $H^{2n-2}(\bar{A},\mathbb{Q}_{\ell}(n-1))$.
Moreover, by the Poincaré Duality \cite[Chap. VI]{Deligne2}, the
cup-product pairing\[
H^{2}(\bar{A},\mathbb{Q}_{\ell}(1))\times H^{2n-2}(\bar{A},\mathbb{Q}_{\ell}(n-1))\to H^{2n}(\bar{A},\mathbb{Q}_{\ell}(n))\simeq\mathbb{Q}_{\ell}\]
 is perfect. Combining these two assumptions and the fact that cohomology
class $\Theta^{n-2}$ is divisible by $(n-2)!$, we get that the cycle
$\frac{1}{(n-2)!}\Theta^{n-2}$ is non-degenerate.
\end{proof}
Let $S$ be a smooth surface over $k$ with a $k$-rational point
$s_{0}$. Let $A$ be the Albanese variety of $S$ and let $\vartheta:S\to A$
be the Albanese map such that $\vartheta(s_{0})=0$. 
\begin{prop}
\label{pro:Iso de galois modules}Suppose that the image $W$ of $\bar{S}$
is a non-degenerate cycle in $\bar{A}$ and $b_{2}(\bar{S})=b_{2}(\bar{A})$.
The map \[
\vartheta^{*}:H^{2}(\bar{A},\mathbb{Q}_{\ell})\to H^{2}(\bar{S},\mathbb{Q}_{\ell})\]
is an isomorphism of Galois modules. The surface $S$ satisfies the
Tate conjecture and $\rho_{S}=\rho_{A}$, where $\rho_{Z}=\dim_{\mathbb{Q}_{\ell}}A_{1}(Z)$
for a geometrically smooth irreducible variety $Z_{/k}$. \end{prop}
\begin{proof}
Let $f:Y\to X$ be a proper map of smooth complete separated varieties
over an algebraically closed field. Let be $a=\dim(X)$, $d=\dim(Y)$
and $c=d-a$. By \cite[Remark 11.6 (d)]{MilneEtale}, there is a linear
map \[
f_{*}:H^{r}(Y,\mathbb{Z}_{\ell})\to H^{r-2c}(X,\mathbb{Z}_{\ell})\]
satisfying the projection formula:\[
f_{*}(y\cdot f^{*}x)=f_{*}y\cdot x,\,\, x\in H^{r}(X,\mathbb{Z}_{\ell}(d)),\, y\in H^{s}(Y,\mathbb{Z}_{\ell}).\]
Let $W$ be the image of $S$ in its Albanese variety $A$. Using
the projection formula, we have \[
\eta_{\bar{A}}(x\cdot y\cdot\vartheta_{*}S)=\eta_{\bar{S}}(\vartheta^{*}x\cdot\vartheta^{*}y\cdot S)=\eta_{\bar{S}}(\vartheta^{*}x\cdot\vartheta^{*}y)\]
for $x,y\in H^{2}(\bar{A},\mathbb{Q}_{\ell}(1))$, and we obtain the
following equality: \[
\eta_{\bar{S}}(\vartheta^{*}x\cdot\vartheta^{*}y)=(\deg\vartheta)\eta_{\bar{A}}(x.W.y),\]
where $\deg\vartheta\not=0$ is the degree of $\vartheta$ onto its
image $W$. Since $Q_{W}(x,y)=\eta_{\bar{A}}(x.W.y)$ and $Q_{W}$
is a non-degenerate pairing, the map \[
\vartheta^{*}:H^{2}(\bar{A},\mathbb{Q}_{\ell}(1))\to H^{2}(\bar{S},\mathbb{Q}_{\ell}(1))\]
is injective. As $b_{2}(\bar{S})=b_{2}(\bar{A})$, the map $\vartheta^{*}$
is then an isomorphism of Galois modules and $H^{2}(\bar{S},\mathbb{Q}_{\ell}(1))\simeq H^{2}(\bar{A},\mathbb{Q}_{\ell}(1))$.
Since the Tate conjecture is satisfied for divisors on Abelian varieties
over the field $k$ of finite type over the prime field, we have $\rho_{A}=\dim H^{2}(\bar{A},\mathbb{Q}_{\ell}(1))^{G}$.
Since $\vartheta^{*}$ is injective, we have $\rho_{S}\geq\rho_{A}$.
On the other hand, for every variety $X$ over $k$, we have $\dim H^{2}(\bar{X},\mathbb{Q}_{\ell}(1))^{G}\geq\rho_{X}$
therefore: \[
\dim H^{2}(\bar{S},\mathbb{Q}_{\ell}(1))^{G}\geq\rho_{S}\geq\rho_{A}=\dim H^{2}(\bar{A},\mathbb{Q}_{\ell}(1))^{G},\]
we thus obtain $\rho_{S}=\dim H^{2}(\bar{S},\mathbb{Q}_{\ell}(1))^{G}=\rho_{A}$
and the Tate conjecture holds for $S$.
\end{proof}
Let us suppose that the field $k$ has characteristic not $2$. Let
$X$ be a smooth cubic hypersurface defined over the field $k$ and
let $S$ be its Fano surface. The surface $S$ is a smooth geometrically
connected variety defined over $k$ \cite[Theorem 1.16 i and (1.12)]{Altman}. 

Let us suppose that the cubic $X$ contains a $k$-rational line $L_{0}$
such that for every line $L'$ (defined over $\bar{k}$) in $X$ meeting
$L_{0}$, the plane containing $L$ and $L_{0}$ cuts out on $\bar{X}$
three distinct lines. Proposition (1.25) in \cite{MurreCompo1} ensures
that such a line $L_{0}$ exists on a finite extension of $k$. Since
the Tate conjecture for $S$ over $k$ holds if and only if it holds
over any finite extension of $k$ (see \cite[Theorem 2.9]{Tate2}),
this assumption on the existence of $L_{0}$ is not a restriction.

The Albanese variety $A$ of $S$ is defined over $k$ \cite[Lemma 3.1]{Achter}
and is $5$ dimensional. Let $\vartheta:S\to A$ be the Albanese map
such that $\vartheta(s_{0})=0$, where $s_{0}$ is the point of the
Fano surface corresponding to $L_{0}$. Let $\Theta$ be the (reduced)
image of $S\times S$ by the map $(s_{1},s_{2})\to\vartheta(s_{1})-\vartheta(s_{2})$.
The variety $\Theta$ is a divisor on $A$ defined over $k$ and $(A,\Theta)$
is a principally polarized Abelian variety (\cite[Proposition 5]{Beauville}
; we checked that although \cite{Beauville} deals with an algebraically
closed field, the assumption on the existence of $L_{0}$ ensures
that it remains true for the field $k$). The Albanese map $\vartheta:S\to A$
is an embedding and the class of $\bar{S}$ in $\bar{A}$ is $\frac{1}{3!}\Theta^{3}$
\cite[Corollaire of §4, and Proposition 7]{Beauville}. Moreover,
by \cite[p. 11]{Bombieri}, $b_{2}(\bar{S})=b_{2}(\bar{A})=45$. 

We thus see that the Fano surface $S$ of $X$ satisfies the hypothesis
of Proposition \ref{pro:Iso de galois modules} and therefore Theorem
\ref{thm:principal result The-Tate-conjecture} holds.

\bigskip

\noindent Xavier Roulleau,
Universit\'e de Poitiers,
Laboratoire de Math\'ematiques et Applications, UMR 7348 du CNRS,
 Boulevard Pierre et Marie Curie,
T\'el\'eport 2 - BP 30179,
86962 Futuroscope Chasseneuil,
France\\
{\tt Xavier.Roulleau@math.univ-poitiers.fr}\\ 

\end{document}